\documentclass[a4paper,11pt]{article}
\usepackage{amsmath}                    % From AMS/CTAN distribution
\usepackage{amssymb}                    % From AMS/CTAN distribution
\usepackage{amsthm}                     % From AMS/CTAN distribution
\usepackage{amsfonts}                   % From AMS/CTAN distrib
\usepackage{subfigure}
\usepackage{color}
\usepackage{graphics,graphicx}
\usepackage{enumerate}
\usepackage{epsfig}
\usepackage{GGGraphics}

\usepackage{multicol}

\allowdisplaybreaks

\newtheorem{theorem}{Theorem}[section]

\newtheorem{remark}[theorem]{Remark}
\newtheorem{ass}[theorem]{Assumption}

\def\ts{\thinspace}

\def\qed{\hfill $\square$ \goodbreak \bigskip}

\DeclareMathOperator{\ddiv}{div}

\renewcommand{\epsilon}{\varepsilon}
\newcommand{\eps}{\varepsilon}

\newcommand{\R}{\mathbb{R}}
\newcommand{\bigo}{\mathcal{O}}
\newcommand{\CC}{\mathcal{C}}
\newcommand{\cN}{\mathcal{N}}

\def \IR{\mathbb R}

\def \IE{\mathbb E}

\def \eps{\epsilon}

\title{
Postprocessed integrators for the high order \\
integration of ergodic SDEs
}

\author{ 
Gilles Vilmart\textsuperscript{1}
}

\date{November 7, 2014}

\begin{document}
\footnotetext[1]{
Universit\'e de Gen\`eve, Section de math\'ematiques, 2-4 rue du Li\`evre, CP 64, 1211 Gen\`eve 4.
%On leave from \'Ecole Normale Sup\'erieure de Rennes, 
%INRIA Rennes, IRMAR, CNRS, UEB,
%av. Robert Schuman, F-35170 Bruz, France, 
Gilles.Vilmart@unige.ch}

\maketitle

\begin{abstract}

The concept of effective order is a popular methodology in the deterministic literature 
for the construction of efficient and accurate integrators for differential equations over long times.
The idea is to enhance the accuracy of a numerical method by using an appropriate change
of variables called the processor. 
We show that this technique can be extended to the stochastic context
for the construction of new high order integrators for the sampling 
of the invariant measure of ergodic systems. 
The approach is illustrated with modifications of the stochastic $\theta$-method applied to Brownian dynamics,
where postprocessors achieving order two are introduced. 
Numerical experiments, including stiff ergodic systems,
illustrate the efficiency and versatility of the approach.

\smallskip

\smallskip

%65L20
%65N30, 65M60 (74Q10, 35K15).
 \end{abstract}

%\tableofcontents

\section{Introduction}

A popular technique in the deterministic literature for the construction of efficient and high order integrators for differential equations
is the notion of effective order of Butcher~\cite{butcher69teo}.
Inspired by this idea, we introduce in this paper a new methodology for the construction of high order integrators for the invariant measure of ergodic systems of stochastic differential equations (SDEs).
Given a deterministic problem 
\begin{equation} \label{eq:ode}
\frac{dy(t)}{dt}=f(y(t)),
\end{equation}
with a smooth vector field $f:\IR^d\rightarrow \IR^d$, the idea is to search for a computationally efficient integrator $K_h$ for \eqref{eq:ode}, named the kernel,
and perturbations of the identity map $\chi_h$ and $\chi_h^{-1}$, inverse of each other, and called respectively the postprocessor and preprocessors, such that the integrator $y_{n+1}=\Phi_h(y_n)$ defined by
$$
\Phi_h = \chi_h \circ K_h \circ \chi_h^{-1}
$$
has a higher order of accuracy than $K_h$. Considering a constant stepsize $h$ yields
\begin{equation}\label{eq:deteffective}
\Phi_h^N = \chi_h \circ K_h^N \circ \chi_h^{-1},
\end{equation}
which means that $N$ compositions of $\Phi_h$ can be efficiently implemented by applying $N$ steps of the kernel $K_h$ and 
only once the processor maps $\chi_h^{-1},\chi_h$, respectively at the beginning and the end of the integration.
This approach is efficient if the computational cost of the kernel $K_h$ is cheaper than the cost of $\Phi_h$.
It was first introduced in \cite{butcher69teo} as a mean to construct $5$th order explicit Runge-Kutta methods with only $5$ internal stages, corresponding to the number of function evaluations per time step.
A renewed interest for this idea appeared in the context of deterministic geometric numerical integration \cite{SaC94,LeR04,HLW06}, where the use of a constant stepsize is required in general, for the construction of high order integrators with a favourable long time behaviour 
 (see e.g. \cite[Sect.\ts V.3.3]{HLW06} and references therein).

In this paper, we consider a system of (It\^o) stochastic differential equations 
\begin{equation} \label{e:main_sde}
dX(t)=f(X(t))dt+g(X(t))dW(t), \quad X(0)=X_{0},
\end{equation}
where $f:\R^d\rightarrow \R^d$ and $g:\R^d\rightarrow \R^{d\times m}$
are smooth fields assumed globally Lipchitz for simplicity, and $W(t)=(W_1(t),\ldots,W_m(t))^T$ is a standard $m$-dimensional Wiener process.
Under appropriate smoothness and growth assumptions on the fields $f,g$, one can show that
\eqref{eq:Brownian} has ergodic dynamics, i.e. for an arbitrary initial condition $X_0=x$, 
assumed deterministic for simplicity, the time average of the trajectory 
satisfies with probability $1$,
$$
\lim_{T \rightarrow \infty}\frac{1}{T}\int_{0}^{T} \phi(X(t)) = \int_{\IR^{d}}\phi(y) d\mu(y).
$$
for all smooth test functions $\phi$ with derivatives at all orders with polynomial growth.
In addition, in many situations, one can show the following exponential ergodicity property for all $t\geq 0$,
\begin{equation} \label{eq:expergo}
\left|\IE(\phi(X(t))) - \int_{\IR^d} \phi(y) d\mu(y) \right| \leq K(x)e^{-\lambda t} 
\end{equation}
where $\lambda>0$ is independent of $x,t$ and $K(x)$ depends on $\phi$ and the initial condition $X_0=x$ (assume deterministic for simplicity), but is independent of $t$.

Consider a one-step numerical integrator for \eqref{e:main_sde},
\begin{equation}\label{eq:methnum}
X_{n+1} = \Psi(X_{n},h,\xi_{n}).
\end{equation}
We say that numerical method $X_n$ for the approximation of \eqref{eq:Brownian} at time $t_n=nh$ is
ergodic if it has a unique invariant probability law $\mu^h$ with finite moments of any order and with probability $1$,
\begin{equation} \label{eq:ergodic_num}
\lim_{N \rightarrow \infty} \frac{1}{N+1}\sum_{n=0}^{N} \phi(X_{n})=\int_{\R^d}\phi(y)d\mu^h(y),
\end{equation}
for all initial condition $X_0=x$ and all test functions $\phi$.
We say that the scheme has order $r\geq1$ with respect to the invariant measure if
%\modg{$|e(\phi,h)| \leq C h^{r}$ with}
\begin{equation} \label{eq:difference1i}
%e(\phi,h):=\left|\lim_{N \rightarrow \infty} \frac{1}{N+1}\sum_{n=0}^{N} \phi(X_{n})-\int_{\EE}\phi(y)d\mu(y)\right|\leq C h^{r}.
|e(\phi,h)| \leq C h^{r}
\quad\mbox{with} \quad
e(\phi,h) := \lim_{N \rightarrow \infty} \frac{1}{N+1}\sum_{n=0}^{N} \phi(X_{n})-\int_{\R^d}\phi(y)d\mu(y),
\end{equation} 
where $C$ is independent of $h$ small enough and $X_0$.
Analogously to \eqref{eq:expergo}, in many situations, one can show for all $t_n=nh$ and all initial condition $X_0=x$, the estimate
\begin{equation} \label{eq:difference1iexp}
\left|\IE(\phi(X_n)) - \int_{\IR^d} \phi(y) d\mu(y)\right| \leq K(x)e^{-\lambda t_n} + C h^r
\end{equation} 
where $C,K(x)$ are independent of $n,h$.

A standard approach for the construction of integrators with high order 
for the invariant measure is to design methods with high weak order of accuracy. This is known for large classes of ergodic SDEs 
in particular for a degenerate noise and locally Lipschitz fields in
\cite{MSH02}.
This can be combined with extrapolation techniques from 
Talay \cite{TaT90}.
It is known, however, that high order can still be achieved using low weak order methods, as shown in particular for certain classes of splitting methods for nonlinear Langevin dynamics \cite{LM13,LMS13,BRO10,AVZ14b}.
We also mention the exact schemes based on Markov Chain Monte-Carlo methods, see e.g. the survey \cite{San14}.
An approach proposed in \cite{AVZ14a,AVZ14b} permits 
to design methods with a high order for the invariant measure despite a low deterministic order of accuracy. 
It relies on recent developments in the stochastic context \cite{ACV12,Zyg11,DeF12,Kop13a,Kop13b} of the theory of backward error analysis and modified differential equations, a fundamental tool in the field of deterministic geometric numerical integration \cite{SaC94,LeR04,HLW06}. 
In the same spirit, the proposed approach of postprocessed integrators, based on Butcher's effective order \cite{butcher69teo},
relies on many ideas from \cite{AVZ14a}.

We shall illustrate our approach at the example of Brownian dynamics generated by the overdamped Langevin equation
\begin{equation} \label{eq:Brownian}
dX({t})=-\nabla V(X(t))dt+ \sigma dW(t),
\end{equation}
where $V : \IR^d \rightarrow \R$ is a smooth potential, $\sigma>0$ is a constant, 
and $W=(W_1,\ldots,W_d)^T$ is a standard $d$-dimensional Wiener process.
This model describes the motion of a particle in a potential subject to thermal noise \cite{RIS84,GAR85}. Under appropriate 
assumptions on the potential $V$, the system is ergodic and has a unique
 invariant measure $\mu$ which is characterized by the Gibbs density function (sometimes called the Boltzmann density function)
\begin{equation} \label{eq:rhoexp0}
\rho(x) = Ze^{-\frac{2}{\sigma^2}V(x)},
\end{equation}
where $Z$ is a renormalization constant such that $\int_{\IR^d} \rho(y) dy = 1$. 
We show that
our approach is closely related to an efficient non-Markovian method of
order two proposed in \cite{LM13} for \eqref{eq:Brownian},
with convergence properties first analyzed in \cite{LMT14}. 
It was obtained considering the large friction limit in a particular splitting method for Langevin dynamics. 
We show in this paper that this scheme
can be interpreted as Markovian up to a change of variable, which is precisely the postprocessor map.

This paper is organized as follows. In Section \ref{sec:meth}, we introduce our new
methodology to extend Butcher's effective order technique to the stochastic context and we provide examples of postprocessed integrators.
In Section \ref{sec:prelim}, we recall standard assumptions that guaranty
the ergodicity of the system and we state the assumptions on the numerical methods needed for our analysis.
In Section \ref{sec:main}, we state our main results and derive conditions on a postprocessor to achieve high order for the invariant measure of a given integrator with emphasis on the case of Brownian dynamics.
Finally, the Section \ref{sec:num} is devoted to the numerical experiments,
and shows the efficiency and versatility of the approach for various stiff and nonstiff Brownian dynamics models.

\section{Postprocessed integrators}
\label{sec:meth}

In this section, we describe the idea of postprocessed integrators for an ergodic system of SDEs \eqref{e:main_sde}. 
Given a smooth test function $\phi$, we recall that there are two standard strategies when computing ergodic averages $\int_{\R^d}\phi(y)d\mu(y)$
using a one step integrator \eqref{eq:methnum} to compute numerical realizations of \eqref{e:main_sde}.
\begin{enumerate}
\item Given an initial condition $X_0$, one can consider a single 
numerical trajectory $X_0,X_1,X_2,\ldots$ over a long time interval of size $T=Nh$, and then compute the approximation
$$
\frac{1}{N+1}\sum_{k=0}^{N} \phi(X_{k}) \simeq \int_{\R^d}\phi(y)d\mu(y),
$$
taking advantage of the estimate \eqref{eq:difference1i}.
\item Alternatively, one can compute a large number $M$
of independent numerical trajectories $\{X_0^i,X_1^i,X_2^i,\ldots,X_N^i\}$, $i=1\ldots,M$ 
on a relatively shorter time interval $T=Nh$
and use the approximation
$$
\frac{1}{M}\sum_{i=1}^{M} \phi(X_N^i) \simeq \int_{\R^d}\phi(y)d\mu(y),
$$
based on the exponential estimate \eqref{eq:difference1iexp}.
The advantage of this second approach, as shown in \cite{MT07} in the context of ergodic SDEs, is that it permits to treat easily the case of non-globally Lipchitz vector fields using standard integrators (in particular explicit ones) by using \textit{the concept of rejecting exploding trajectories}, a rigorous methodology introduced in \cite{MT05}.
This procedure is often preceded by an equilibration period to make the law of $X_0^i$ close to the invariant measure $\mu$, see e.g. \cite{MT07} for details.
\end{enumerate}

The idea of postprocessed integrators is to consider an
integrator \eqref{eq:methnum} with order $r$ for the invariant measure in \eqref{eq:difference1i} and \eqref{eq:difference1iexp},
and to search for a 
postprocessor $\overline X_n = G_n(X_n)$, which can be a deterministic or a stochastic perturbation of $X_n$, such that $\overline X_n$
has a higher order of accuracy $\overline r>r$ for the invariant measure. Precisely, for all  $h$ assumed small enough, and all $t_n=nh$,
\begin{eqnarray} \label{eq:difference1iG}
\left| \lim_{N \rightarrow \infty} \frac{1}{N+1}\sum_{k=0}^{N} \phi(\overline X_{k})-\int_{\R^d}\phi(y)d\mu(y) \right| &\leq& C h^{\overline r}
\\
\left|\IE(\phi(\overline X_n)) - \int_{\IR^d} \phi(y) d\mu(y) \right| &\leq& K(x)e^{-\lambda t_n} + C h^{\overline r}, \label{eq:difference1iexpG}
\end{eqnarray}
where $C,\lambda$ are independent of $h,n$ and the initial condition $X_0=x$,
and $K(x)$ is independent of $h,n$.

If ergodic averages are computed using the estimate \eqref{eq:difference1iG},
then the postprocessor $\overline X_k=G_k(X_k)$ is applied at every time step $k$, but it has a negligible overcost in some situations where no additional function evaluations are needed as shown below. In addition $\overline X_k$ can always be computed in parallel with $X_{k+1}$, independently.
If alternatively, ergodic averages are approximated based on the exponential estimate \eqref{eq:difference1iexpG}, then the postprocessor $\overline X_N=G_N(X_N)$ is applied only once at the end time $t_N=Nh$ of each numerical trajectory,  resulting in a negligible overhead.

\begin{remark}
Compared to the deterministic setting \eqref{eq:deteffective},
$G_n$ plays the role of the postprocessor map $\chi_h$. In contrast, notice however that no preprocessor $\chi_h^{-1}$ is involved in the stochastic setting. This is not surprising because the limit $\IE(\phi(\overline X(t))) \rightarrow \int_{\IR^d} \phi(y) \mu(y)dy$ for $t\rightarrow \infty$ is independent of the initial condition $X_0=x$. In other words, the ergodic 
systems ``forgets'' the initial condition. The advantage compared to the deterministic case is that the postprocessor map $G_n$ needs not to be reversible, it can be  in particular the  flow of a stiff deterministic problem, or even a stochastic problem.
\end{remark}

To illustrate our new approach, we focus on the problem of Brownian dynamics
\eqref{eq:Brownian}, and  we consider the standard stochastic $\theta$-method
\begin{equation} \label{eq:thetastd}
X_{n+1} = X_n + h  (1-\theta)f(X_n) 
+ h \theta f(X_{n+1}) + \sigma \sqrt{h} \xi_n,
\end{equation}
where $f=-\nabla V$ and $\xi_n\sim\cN(0,I)$ are independent Gaussian random vectors with dimension $d$.
For $\theta=0$, it reduces to the
Euler-Maruyama method, while for $\theta\ne 0$, the method is implicit and requires 
in general the resolution of a non-linear system at each timestep.
For $\theta\geq 1/2$, it is shown in \cite{Hig00} that the $\theta$-method is mean-square $A$-stable, a property which makes it well-suited for stiff 
mean-square stable stochastic problems. 
This scheme has weak order one of accuracy, except in the special case $\theta=1/2$ for which the weak order of the method is two. 
Is was shown in \cite{Sch99b} that the $\theta$-method preserves exactly the invariant measure in the linear case if and only if $\theta=1/2$ (i.e. for $V(x)$ a quadratic function in \eqref{eq:Brownian}, which corresponds to an Orstein-Uhlenbeck process).
This result was extended in \cite{BuB12} for linear second order problems with additive noise, where it was shown that the $\theta=1/2$-method is in fact the only measure exact scheme in a certain class of stochastic Runge-Kutta methods involving only one Wiener increment per step. 

For multi-dimensional nonlinear problems of the form \eqref{eq:Brownian}, we introduce the following postprocessed methods of order $2$ for the invariant measure based on modifications of the $\theta$-method \eqref{eq:thetastd}. Precisely, we prove in Section \ref{sec:main} that 
for each method below, $\overline X_n$ has order $\overline r=2$ for the invariant measure in \eqref{eq:difference1iG},\eqref{eq:difference1iexpG}. 
%\begin{itemize}
%\item 
For $\theta=0$, we obtain the following modification of the Euler-Maruyama method, which is equivalent to a method first proposed in  \cite{LM13}  (see \eqref{eq:non-Markovian} below)
\begin{equation} \label{eq:non-Markovian0}
X_{n+1} = X_n + h  f\left(X_n + \frac12 \sigma \sqrt{h} \xi_n\right) 
 + \sigma \sqrt{h} \xi_n,
\qquad
\overline X_n=  X_n + \frac12 \sigma \sqrt{h} \xi_n.
\end{equation}
%\item 
For $\theta=1$, we derive the new scheme
\begin{equation} \label{eq:theta1}
X_{n+1} = X_n + h  f(X_{n+1} + a \sigma \sqrt{h} \xi_n) 
 + \sigma \sqrt{h} \xi_n,\quad
a= -\frac{1}2+\frac{\sqrt{2}}2,
\end{equation}
together with two possible postprocessors
\begin{equation} \label{eq:poststo}
\overline X_n=  X_n + c \sigma \sqrt{h} J_n^{-1}\xi_n,
\quad c= \frac{\sqrt{{2\sqrt{2}}-1}}2
\end{equation}
and alternatively,
\begin{equation} \label{eq:postdet}
\overline X_n=  X_n + h b f(\overline X_{n})  + c \sigma \sqrt{h} \xi_n,\quad
b=\frac{\sqrt{2}}2,\quad c=\frac{\sqrt{{4\sqrt{2}}-1}}2.
\end{equation}
%\end{itemize}
The matrix $J_n^{-1} $ in \eqref{eq:poststo} is defined as the inverse of the $d\times d$ Jacobian
matrix 
$$J_n=I-hf'(X_n+a\sigma \sqrt h \xi_{n-1}).$$

\begin{remark}
The matrix $J_n$ is involved in standard Newton iterations needed 
to solve the nonlinear system \eqref{eq:theta1}. 
For an efficient implementation, it is known that one should never compute the inverse $J_n^{-1}$ itself but rather solve the appropriate linear systems.
Using a precomputed LU decomposition of this matrix, the evaluation of $J_n^{-1}\xi_n$ in \eqref{eq:poststo} therefore has a negligible cost. This term is introduced only to improve the numerical
stability and does not influence the order $2$ of accuracy of the postprocessed method \eqref{eq:theta1},\eqref{eq:poststo}, as discussed in Section \ref{sec:main} (Theorem \ref{thm:thetameth}).
The use of Jacobian matrices such as $J_n$ for stabilization is classical for stiff deterministic problems \cite[Sect.\ts IV.8]{HaW96}, and reveals efficient also in the context of mean-square stable stiff SDEs in \cite{AVZ13b}. 
\end{remark}

%\begin{remark} \label{rem:leimkhuler}
Substituting $X_n=\overline X_n -  \frac12 \sigma \sqrt{h} \xi_n$ in 
\eqref{eq:non-Markovian0}, a calculation shows that the scheme \eqref{eq:non-Markovian0} is equivalent to the method
\begin{equation} \label{eq:non-Markovian}
\overline X_{n+1} = \overline X_{n} + h  f(\overline X_{n}) + \frac12 \sigma \sqrt{h} (\xi_n+\xi_{n+1}).
\end{equation}
The scheme \eqref{eq:non-Markovian} of order two, which has the same cost as the standard Euler-Maruyama method of order one,
was first introduced in \cite{LM13} as a non-Markovian method. 
A rigorous analysis 
of its order two of convergence for the numerical invariant measure of \eqref{eq:Brownian} was first proposed in \cite{LMT14}. Notice that the point of view of postprocessed integrators provides here an alternative
proof of the order two of accuracy of the scheme \eqref{eq:non-Markovian}.
Analogously, notice that a non-Markovian formulation of the postprocessed method \eqref{eq:theta1},\eqref{eq:poststo} can also be written,
$$
\overline X_{n+1}  = \overline X_{n}  
+ h  f(\overline X_{n+1} +  \sigma \sqrt{h} (a\xi_n  - c \xi_{n+1})) 
+ \sigma \sqrt{h} ((1-c) \xi_n + c   \xi_{n+1}),
$$
where $a,c$ are defined in  \eqref{eq:theta1},\eqref{eq:poststo},
and we omit for simplicity the stabilization term~$J_n$. 

\section{Background}
\label{sec:prelim}

We recall in this section standard results and assumptions
that will be useful for the analysis of postprocessed integrators in Section \ref{sec:main}.
We assume that $f,g$ in \eqref{e:main_sde} are smooth fields and we focus for simplicity on the case of globally Lipchitz vector fields.

We denote $\CC_P^\infty(\IR^d,\IR)$ the set of $\CC^\infty$ functions whose derivatives
up to any order have a polynomial growth of the form
\begin{equation} \label{eq:phi_assumption}
|\phi(x)| \leq C (1+|x|^{s})
\end{equation}
for some constants $s$ and $C$ independent of $x$.
A classical fundamental tool for the analysis of \eqref{e:main_sde}
is the backward Kolmogorov equation \cite[Chap.\ts 2]{MiT04}.
Given $\phi \in \CC_P^\infty(\IR^d,\IR)$ and setting $u(t,x)=\mathbb{E}\left( \phi(X(t))|X_{0}=x\right)$, it yields that
$u(x,t)$ is the solution of the deterministic parabolic PDE
%Given $\phi \in \CC_P^\infty(\IR^d,\IR)$ and setting $u(t,x)=\mathbb{E}\left( \phi(X(t))|X_{0}=x\right)$, we have that
%$u(x,t)$ solves the deterministic  PDE
\begin{equation} \label{eq:Kolmogorov}
%\begin{eqnarray} 
\frac{\partial u}{\partial t} = \mathcal{L}u, \qquad
u(x,0) = \phi(x), \qquad x\in\R^d, t>0,
%\end{eqnarray}
\end{equation}
where the generator $\mathcal{L}$ is defined by
\begin{equation} \label{eq:generator}
\mathcal{L}\phi:= f\cdot \nabla\phi+\frac{1}{2}gg^{T}:\nabla^2 \phi,
\end{equation}
where $\nabla \phi$ and $\nabla^2 \phi$ are the gradient vector and the Hessian matrix of $\phi$,
and $A:B=\mathrm{trace}(A^TB)$ is a scalar product on matrices.

\begin{ass}
\label{th:ergodic1}
We assume the following.
%Let $X(t)$ the solution of \eqref{e:main_sde}. 
%The following assumptions implies
%its ergodicity (see \cite{Has80}),
\begin{enumerate}
\item $f,g$ are of class $C^{\infty}$, with bounded derivatives of any order, and $g$ is bounded;
%\item the generator $\mathcal{L}$ in \eqref{eq:generator} is a uniformly elliptic operator or an hypo-elliptic operator, and we assume the uniqueness of the invariant measure of \eqref{e:main_sde} in this latter case;
\item the generator $\mathcal{L}$ in \eqref{eq:generator} is a uniformly elliptic operator, \emph{i.e.} there exists $\alpha>0$
such that for all
$
x, \xi \in \IR^{d} , \ x^T g(\xi) g(\xi)^T x %\sum_{i,j} g_{ij}(\xi)x_{i}x_{j} 
\geq \alpha x^Tx
$;
%\item in the case where $\mathcal{L}$ is a hypo-elliptic operator, we further assume the uniqueness of the invariant measure,
\item there exist $C,\beta>0$ 
such that for all $x\in \R^d$,
$x^T f(x) \leq -\beta x^Tx+C$. 
%and a compact $K \subset \R^d$ such that \modg{for all}
%$
%x \in \IR^{d}\setminus K ,\  \modg{x^T f(x)} \leq -\beta \modg{x^Tx}.
%$  
\end{enumerate}
\end{ass}
Assumption \ref{th:ergodic1} implies automatically
that \eqref{e:main_sde} is ergodic with a unique invariant measure (see e.g. \cite{Has80}) and
satisfies the exponential ergodicity property \eqref{eq:expergo}. 
Notice that hypo-elliptic generators $\mathcal{L}$ could also be considered, to include the case of Langevin dynamics. 
We obtain for all $\phi \in \CC_P^\infty(\IR^d,\IR)$ the identity
\begin{equation} \label{eq:phiL}
\lim_{t\rightarrow \infty} u(x,t) = \phi(x) + \int_0^\infty \mathcal{L}u(t,x)dt =  \int_{\IR^d}\phi(y)\rho(y)dy .
\end{equation}
%We further assume that there
%exists a constant $\lambda$ and for all integer $k\geq 0$ constants $C_k,\kappa_k$ such that for all $t\geq 0$
%\begin{equation}\label{eq:expdecay}
%\|u(t,\cdot) - \int_{\IR^d} \phi(y)\rho_\infty(y)dy\|_{\CC^k}
%\leq C_k(1+t^{\kappa_k}) e^{-\lambda t} \|\phi\|_{\CC^k},
%\end{equation}
%where $\|v(t,\cdot)\|_{\CC^k}$ denotes the norm of the function $v(x,t)$ and its derivatives with respect to $x$ up to order $k$. Notice that setting $t\rightarrow\infty$ in \eqref{eq:expdecay} and using equation \eqref{eq:Kolmogorov} yields
%\begin{equation} \label{eq:phiL}
%\lim_{t\rightarrow \infty} u(x,t) = \phi(x) + \int_0^\infty \mathcal{L}u(t,x)dt =  \int_{\IR^d}\phi(y)\rho_{\infty}(y)dy .
%\end{equation}
%We refer to  \cite{BaT96,MST10,DeF12,Kop13a,Kop13b} for a discussion of the Assumptions
%\ref{th:ergodic1} and \eqref{eq:expdecay}.
Considering $L^{2}$-adjoint of the generator $\mathcal{L}$, given by
\begin{equation} \label{eq:adjoint}
\mathcal{L}^{*}\phi=-\ddiv(f \phi)+\frac{1}{2}gg^{T}:\nabla^2\phi,
\end{equation}
we recall that the density $\rho$ of the invariant measure $\mu$ of \eqref{e:main_sde} is the unique solution of
\begin{equation} \label{eq:Lstar_rho}
\mathcal{L}^{*} \rho = 0.
\end{equation}
We make the following assumptions on the one step integrator \eqref{eq:methnum}. For $\phi\in \CC_P^\infty(\IR^d,\IR)$, consider the quantity
\begin{equation} \label{eq:numin_def}
U(x,h)=\mathbb{E}(\phi(X_{1})|X_0=x)).
\end{equation}
We assume that the scheme has local weak order $p$ i.e. it satisfies for all initial condition $X_0=x$,
\begin{equation}  
\label{weak:conv_loc}
|\IE(\phi(X_1)) -  \IE(\phi(X(h)))| \leq C(x) h^{p+1},
\end{equation}
for all $h$ sufficiently small, where $C(x)$ is independent of $h$ but depends on $x,\phi$. 

\begin{ass} \label{ass1} 
%Let $f,g$ be $C^\infty$ functions on the torus. 
%Let $f,g$ be $C^\infty$ with bounded derivatives of any order. 
We assume that \eqref{eq:numin_def} has a weak Taylor series
expansion of the form,
\begin{equation} \label{eq:numin_taylor_expansion_formal}
U(x,h)=\phi(x)+hA_0\phi(x)+h^2 A_1\phi(x)+\ldots,
\end{equation}
for all $\phi \in \CC_P^\infty(\IR^d,\IR)$,
where $A_i,~i=0,1,2,\ldots$ are linear differential operators with coefficients depending smoothly on the drift and diffusion functions $f,g$, and their derivatives (and depending on the choice of the integrator). In addition, we assume that $A_0$
coincides with the generator $\mathcal{L}$ given in \eqref{eq:generator},
which means that the method has (at least) local order one in the weak sense,
\begin{equation} \label{eq:A0}
A_0(f,g) = \mathcal{L}.
\end{equation}
\end{ass}
For instance, for the $\theta$-method \eqref{eq:thetastd} applied to \eqref{eq:Brownian}, a calculation shows
\begin{eqnarray}\label{eq:A1std}
A_1\phi &=& \frac12 \phi''(f,f) + \frac{\sigma^2}2 \sum_{i=1}^d \phi'''(e_i,e_i, f)
+ \frac{\sigma^4}8  \sum_{i,j=1}^d \phi^{(4)} (e_i,e_i,e_j,e_j) +\theta \phi' f'f\nonumber\\
&+&  \theta \frac{\sigma^2}2\phi' \sum_{i=1}^d f''(e_i,e_i)
+ \theta  \sigma^2\sum_{i=1}^d \phi''(f'e_i,e_i),
\end{eqnarray}
where $e_1,\ldots,e_d$ is the canonical basis of $\R^d$, and $\phi'(x),\phi''(x),\phi'''(x),\ldots$ are the derivatives of $\phi$, which are linear, bilinear, trilinear, \ldots, forms, respectively (and analogously for $f$). We refer to \cite{Zyg11,ACV12} for examples of similar calculations.
We shall also use the following assumption \cite{Mil86} (see \cite[Chap.\ts 2.2]{MiT04}) which permits to guaranty that the numerical moments $\IE(|X_n|^k)$ up to any order $k$ are bounded along time. This also permits to infer automatically ``a global weak order $p$'' from the local weak error
\eqref{weak:conv_loc} of a given integrator, see \cite[Chap.\ts 2.2]{MiT04}.

\begin{ass} \label{ass:moments}
We assume that the numerical integrator \eqref{e:main_sde} satisfies for all $x\in\mathbb{R}^d$ %deterministic initial conditions $X_0=x$ 
that
\begin{equation} \label{ass:Milstein}
|\IE(X_{1}-X_0 | X_0=x)| \leq C (1+|x|) h, \qquad |X_{1}-X_0| \leq M (1+|X_0|) \sqrt h,
\end{equation}
where $C$ is independent of $h$ small enough and $M$ is a random variable that has bounded moments of all orders independent of $h$ and $X_0$. 
\end{ass}

Theorem \ref{thm:talay} below is standard and combines results derived by Talay and Milstein.
Precisely, it has been proved in \cite{TaT90} for specific methods
(the Euler-Maruyama and the Milstein methods), while the general procedure to infer  the global weak order
from the local weak order is due to Milstein \cite{Mil86} and can be found in \cite[Chap.\ts 2.2,\ts 2.3]{MiT04}.
The proof of Theorem \ref{thm:talay} is thus omitted.

\begin{theorem} \cite{TaT90,Mil86}
\label{thm:talay}
%Assume that $f,g$ in \eqref{e:main_sde} are $C^\infty$ with bounded derivatives up to any order. 
%Assume Assumption \ref{ass1} in $\IR^d$ and  
%
%Assume the hypotheses on $f,g$  in Remark \ref{th:ergodic1}. 
%\modg{(removed\eqref{eq:expCT})}
%in Remark \ref{th:numerical_rigorous} and Remark \ref{th:ergodic1}.
Let $X_{N}$ be an ergodic numerical solution of \eqref{e:main_sde} on $[0,T]$ and assume 
Assumption \ref{th:ergodic1}, Assumption \ref{ass1}, Assumption \ref{ass:moments}. Assume further the local weak order $p$ estimate \eqref{weak:conv_loc} where $C(x)$ has a polynomial growth of the form \eqref{eq:phi_assumption}.
%(where with $C(X_0)\leq K(1+|X_0|^\kappa),~\forall X_0\in\mathbb{R}^d$ for some $K,\kappa>0$).
%Then, we have the following expansion of the global error,
%for all $\phi\in C_P^{\infty}(\R^d,\R)$,
%\begin{equation} \label{eq:main_result1}
%\IE(\phi(X(T)))-\IE(\phi(X_N))= h^{p}\int_{0}^{T}\IE(\psi_{e}(X({s}),s))ds+\mathcal{O}(h^{p+1})
%\end{equation}
%where $\psi_{e}(x,t)$ satisfies 
%\begin{equation} \label{eq:error_coefficient}
%\psi_{e}(x,t)=\left(\frac{1}{(p+1)!}\mathcal{L}^{p+1}-{A}_{p}\right)v(x,t),
%\end{equation}
%with $v(x,t)=\IE(\phi(X({T}))|X({t})=x)$ satisfying 
%\begin{equation} \label{eq:Kolomogorov1}
%%\begin{eqnarray}
%\frac{\partial v}{\partial t}+\mathcal{L}v=0, \qquad
%v(x,T)=\phi(x).
%%\end{eqnarray}
%\end{equation}
%\end{theorem}
%
%\begin{theorem} \label{th:difference}
%Assume that the hypotheses of Theorem \ref{thm:talay} and Assumption \ref{th:ergodic1} hold. 
Then, the %if a numerical method of weak order $p$ is ergodic, its 
invariant measure error
in \eqref{eq:difference1i} satisfies
for all $\phi\in\CC_P^\infty(\IR^d,\IR)$ and $h\rightarrow 0$,
\begin{equation} \label{eq:difference}
e(\phi,h)
=\lambda_{p}h^{p}+\mathcal{O}(h^{p+1})
\end{equation}
for any deterministic initial condition,  with $\lambda_{p}$ defined as 
\begin{equation} \label{eq:lambda}
\lambda_{p}=\int_{0}^{+\infty}\int_{\IR^{d}}\left({A}_{p}-\frac{1}{(p+1)!}\mathcal{L}^{p+1}\right)u(y,t) \rho(y)dydt
\end{equation}
where $u(x,t)$ is the solution of \eqref{eq:Kolmogorov}. 
In addition, there exists constants $\lambda,C,K(x)>0$ such that for all $n\geq 0$,
\begin{equation}\label{eq:numexp}
\left|\IE(\phi(X_n)) - \int_{\IR^d} \phi(y) d\mu(y) - \lambda_p h^p \right|
\leq K(x)e^{-\lambda t_n} + Ch^{p+1}
\end{equation}
where $t_n=nh$, the constants $C,K(x),\lambda$ are independent of $n,h$, and $h$ is assumed small enough ($C,K(x)$ depend on $\phi$ and
$K(x)$ depends on $x$).
\end{theorem}

In fact, the assumption that the integrator has weak order $p$ in Theorem \ref{thm:talay} can be relaxed, as explain in the following remark. 

\begin{remark} \label{rem:star}
Using $\mathcal{L}^*\rho=0$ and considering the $L^2$-adjoint\footnote{%
%Recall that the $L^2$-adjoint ${A}_{p}^*\rho$ is defined via the identity
%$
%\int_{\mathbb{R}^d} {A}_{p} \phi(x) \psi(x) dx = \int_{\mathbb{R}^d}\phi(x)  {A}_{p}^* \psi(x) dx
%$
%for all $\phi,\psi\in\CC_P^\infty(\IR^d,\IR)$. 
For instance, for the Euler-Maruyama method, setting $\theta=0$ in \eqref{eq:A1std} and using the notation~$\Delta$ for the  Laplace operator, we obtain 
$A_1 \phi = \frac12 \phi''(f,f) + \frac{\sigma^2}2 (\Delta \phi)' f + \frac{\sigma^4}{8} \Delta^2 \phi$  and
$$A_1^* \phi = \frac12 \ddiv\left(\phi (f'f + \ddiv (f) f) + \|f\|^2 \nabla \phi - \sigma^2  f \Delta \phi\right)    + \frac{\sigma^4}{8} \Delta^2 \phi.$$
}
${A}_{p}^*$ of the linear differential operator ${A}_{p}$, we deduce that $\lambda_{p}$ in \eqref{eq:lambda} satisfies
$$
\lambda_{p}=\int_{0}^{+\infty}\int_{\IR^{d}} u(y,t) {A}_{p}^*\rho(y)dydt.
$$
%where $A_p^*$ is the $L^2$ adjoint of the linear differential operator ${A}_{p}$.
We see that if ${A}_{p}^*\rho=0$, then the scheme has order $p+1$ for the invariant measure.
More generally, it is shown in \cite{AVZ14a} based on backward error analysis results in \cite{DeF12} on the torus
and generalizations \cite{Kop13a,Kop13b} in $\IR^d$ 
that the (local) weak order $p$ assumption \eqref{weak:conv_loc} in Theorem \ref{thm:talay} (equivalent to $A_{j+1}={\mathcal L^j}/{j!}$ for all $j<p$ in Assumption \ref{ass1})
can be replaced by the more general condition
\begin{equation} \label{eq:general_condition}
A^{*}_{j}\rho=0, \quad \text{for}  \quad j=1,\cdots p-1.
\end{equation}
This permits to derive in \cite{AVZ14a} new high order integrators based on modified differential equations
while it serves in \cite{AVZ14b} as a tool for the analysis of the order of convergence of Lie-Trotter splitting methods
applied to Langevin dynamics.
\end{remark}

\section{Analysis of postprocessed integrators}
\label{sec:main}

This section is devoted to the analysis of postprocessed integrators.
To illustrate our approach, we first study the simplest linear case of the
Orstein-Uhlenbeck process in dimension one.
%\footnote{%
%Note that a $d$-dimensional linear problem \eqref{eq:Brownian} with quadratic potential $V(x)$ is equivalent to 
%$d$~independent scalar linear problems by considering the basis of eigenvectors of the Hessian $\nabla^2 V$.}
We then focus on the analysis for nonlinear problems, with special emphasis
on the case of Brownian dynamics.

\subsection{Accuracy and stability analysis in the linear case}

We consider for simplicity the Orstein-Uhlenbeck process in dimensions $d=m=1$,
\begin{equation} \label{eq:OU}
dX=-\gamma X dt + \sigma dW(t),
\end{equation}
where $\gamma,\sigma>0$ are fixed constants. It corresponds to Brownian dynamics \eqref{eq:Brownian} with $V(x) = \frac\gamma2 x^2$.
For all initial condition $X_0=x$, the exact solution $X(t)$ is a Gaussian random variable with $\lim_{t\rightarrow  \infty} \IE(X(t)^2)=\frac{\sigma^2}{2\gamma}$.
Indeed, this system is ergodic and its invariant measure is Gaussian,
with density 
$$\rho(x)= \frac1{\sigma\sqrt{2\pi}}%\sigma^{-1}({2\pi})^{-1/2}
e^{-\frac{\gamma}{2\sigma^2} x^2}.$$
Since the invariant measure is Gaussian, the error is obtained simply by comparing the exact and numerical variances.
Applying a one step numerical method to \eqref{eq:OU}, in particular a Runge-Kutta method, yields
in general a recursion of the form
\begin{equation} \label{eq:rklin}
X_{n+1} = A(z) X_n + B(z) \sqrt h \sigma \xi_n,
\end{equation}
where $z=-\gamma h$ and $A(z)=1+z+\bigo(z^2)$ and $B(z)=1+\bigo(z)$ are analytic functions depending on the choice of the method and $\xi_n\sim \mathcal{N}(0,1)$
are independent Gaussian random variables.
Then, a calculation using the recursion
$\IE(X_{n+1}^2)=\IE(X_n^2)A(z)^2 + B(z)^2 \sigma^2 h$
yields
\begin{equation} \label{eq:defR}
\lim_{n \rightarrow \infty} \IE(X_n^2) = \frac{\sigma^2}{2\gamma} |R(z)|,\qquad \mbox{where } R(z)=\frac{2z B(z)^2}{A(z)^2-1},\quad z=-\gamma h,
\end{equation}
for all $h$ small enough such that $|A(z)|<1$ (stable stepsize), and $\lim_{n \rightarrow \infty} \IE(X_n^2) = \infty$ if $|A(z)|> 1$ and
$\IE(X_0^2) \ne 0$  (unstable stepsize).

\paragraph{Stability for linear problems}
We recall here useful standard stability results. The scalar
linear SDE \eqref{eq:OU} can be considered as a test equation which gives
insight on the behavior for nonlinear problem with additive noise. Notice that other test equations are used in the literature in the case of multiplicative noise, see \cite{SaM96,Hig00,SaM02,BBT04,Toc05,RaB08,BuC10}.

We see from the above discussion that $\IE(X_n^2)$ remains bounded as
$n\rightarrow \infty$ (mean-square stability) if $|A(z)|<1$.
For the $\theta$ method, it can be checked \cite{Hig00,Hig00b} that 
this is satisfied for all stepsize $h>0$ if and only if $\theta \geq 1/2$, 
while for the Euler-Maruyama method ($\theta=0$), we have the stepsize restriction
$h\gamma <2$ which can be a severe restriction for large values of $\gamma$.
This stability property for all complex $\gamma$ with positive real part, unconditionally on the stepsize $h$, is called $A$-stability in the framework of deterministic Runge-Kutta methods \cite{HaW96}.
This makes the $\theta$-method with $\theta\geq 1/2$
well suited for stiff stochastic problems, in contrast to the Euler-Maruyama method.
We have the additional feature $\lim_{\gamma\rightarrow +\infty} \IE(X_1^2) =0$ for all $X_0=x$ if and only if $\theta=1$, a property called $L$-stability in the setting of deterministic Runge-Kutta-methods ($\sigma=0$), and that is also important in the context of mean-square stable stiff SDEs \cite{AVZ13b}. The $L$-stability property is desirable for severely stiff problems with very large eigenvalues on the negative real axis, because it permits to damp the higher modes. We emphasis that the $L$-stability property $\lim_{\gamma\rightarrow +\infty} \IE(\overline X_1^2) =0$ is also satisfied for the postprocessed methods \eqref{eq:theta1},\eqref{eq:poststo} and \eqref{eq:theta1},\eqref{eq:postdet}.

\paragraph{Accuracy analysis}

We see from \eqref{eq:defR} that the method (without postprocessor) has order $p$ for the Gaussian invariant measure if and only if $R(z)=1+\bigo(z^p)$ for $z\rightarrow 0$.
Applying a corrector $\overline X_n = G_n(X_n)$  of the form
$$
G_n(x) = C(z)  x + D(z) \sqrt h \sigma \xi_n,\quad z=-\gamma h,
$$
where $C(z)=1+\bigo(z)$ does not affect the mean-square stability of $\overline X_n$, we obtain
$$
\lim_{n \rightarrow \infty} \IE(\overline X_n^2) = 
%\frac{\sigma^2}{2\gamma} %\frac{ B_h^2 C_h^2}{A_h(1-\gamma h A_h/2)}
%C(z)^2\lim_{n \rightarrow \infty} \IE(X_n^2) 
 %+ h\sigma^2 D(z)^2= 
\frac{\sigma^2}{2\gamma} C(z)^2|R(z)|+ h\sigma^2 D(z)^2,\quad z=-\gamma h.
%= \frac{\sigma^2}{2\gamma} \frac{\gamma^2 B_h^2 C_h^2 + 2\gamma D_h^2 A_h(1-\gamma h A_h/2)}{A_h(1-\gamma h A_h/2)}
$$
Comparing with the exact solution for which $\lim_{t \rightarrow \infty} \IE(X(t)^2)=\frac{\sigma^2}{2\gamma}$,
we deduce that the method has order $p$ for the invariant measure of the Orstein-Uhlenbeck process if and only if
the above limit is $\frac{\sigma^2}{2\gamma} + \bigo(h^p)$, equivalently
\begin{equation} \label{eq:condlin}
R(z)C(z)^2 - 
2zD(z)^2
= 1+\bigo(z^p),\quad z\rightarrow 0.
\end{equation}
%$$
%\frac{\sigma^2}{2\gamma} \frac{ B_h^2 C_h^2}{A_h(1-\gamma h A_h/2)} + h\sigma^2 D_h^2 - \frac{\sigma^2}{2\gamma} = \bigo(h^p).
%$$
It turns out that the above identity holds exactly for the postprocessed scheme \eqref{eq:non-Markovian0} for which $A(z)=1+z$, $B(z)=1+z/2$, $C(z)=1$, $D(z)=1/2$.
This yields the following remark
which states that the explicit postprocessed scheme \eqref{eq:non-Markovian0} (equivalently \eqref{eq:non-Markovian} from \cite{LM13}) has no bias in sampling the invariant measure for all stepsize $h$, a feature in the linear case shared with the implicit $\theta=1/2$-method \cite{Sch99b}. 

\begin{remark}
The postprocessed scheme \eqref{eq:non-Markovian0} applied to Brownian dynamics \eqref{eq:Brownian} samples exactly the invariant measure in the case of a linear drift (i.e. if $V(x)$ is a quadratic function on $\R^d$). 
\end{remark}

In \cite{BuL09}, in the context of linear second order problems, new $s$-stage explicit Runge-Kutta methods are introduced with high order  in the velocity correlation matrix with accuracy $\bigo(h^s)$ for $s=3,4,5$, respectively.
In the same spirit, we make the following remark which states using \eqref{eq:condlin} that
there always exist postprocessors achieving arbitrarily high order for
sampling the invariant measure for all consistent integrators applied to the Orstein-Uhlenbeck process.
\begin{remark}
Given a consistent method of the form \eqref{eq:rklin} for \eqref{eq:OU}, for arbitrary integer $s$, 
there exists a deterministic postprocessor of the form $\overline X_n = C(-\gamma h)X_n$ where $C(z)=R(z)^{-1/2} + \bigo(z^s)$ is a polynomial of degree strictly less than $s$ (corresponding 
in particular to an $(s-1)$-stage explicit Runge-Kutta method) such that the corresponding postprocessed scheme has order $s$ for the invariant measure. 
For instance, for the Euler-Maruyama method (see \eqref{eq:thetastd} with $\theta=0$), we obtain $C(z)=\sqrt{1+z/2}=1+z/4-z^2/32 + \ldots$.
\end{remark}

\subsection{Postprocessed integrators for nonlinear ergodic SDEs}

We may now state the main result of this paper for the analysis of
postprocessed integrators.
We shall use the commutator notation $[A,B]=AB-BA$ on linear differential operators.

\begin{theorem} \label{thm:main}
Assume the hypotheses of Theorem \ref{thm:talay} and 
consider a numerical method $\{X_{n}\}_{n\geq 0}$ with weak order $p$. 
Let $G_n$ denote independent and identically distributed random maps in $\R^d$,  independent of $\{X_{j}\}_{j\leq n}$
and satisfying \eqref{ass:Milstein} with $X_1$ replaced by $G_1(X_0)$ and a weak Taylor expansion of the form
\begin{equation} \label{eq:phiG}
\IE(\phi(G_n(x))) = \phi(x) + h^p \overline A_p \phi(x) + \bigo(h^{p+1}),
\end{equation}
for all $\phi \in \CC_P^\infty(\IR^d,\IR)$, where the constant in $\bigo$ has a polynomial growth with respect to $x$ of the form \eqref{eq:phi_assumption}.
Assuming further
\begin{equation} \label{eq:asscommutator}
({A}_{p} + [\mathcal{L},\overline A_p])^*\rho = 0,
\end{equation}
then $\overline X_n = G_n(X_n)$ yields an approximation of order $p+1$ for the invariant measure, precisely 
\begin{eqnarray} \label{eq:th1}
\left| \lim_{N \rightarrow \infty} \frac{1}{N+1}\sum_{k=0}^{N} \phi(\overline X_{k})-\int_{\R^d}\phi(y)d\mu(y) \right| &\leq& C h^{p+1}
\\ \label{eq:th2}
\left|\IE(\phi(\overline X_n)) - \int_{\IR^d} \phi(y) d\mu(y) \right| &\leq& K(x)e^{-\lambda t_n} + C h^{p+1}, 
\end{eqnarray}
for all $\phi\in\CC^\infty_P(\IR^d,\IR)$ with $t_n=nh$, where $C,K(x)$ are independent of $n$ and $h$ assumed small enough but depend of $\phi$ (and $K(x)$ depends on the initial condition $X_0=x$). 
\end{theorem}

\begin{proof}
Using \eqref{eq:phiG} yields
\begin{equation*}
\IE \left(\int_{\IR^d} \phi(G_n(y)) \rho(y) dy\right) = \int_{\IR^d} \phi(y) \rho(y) dy + h^p \int_{\IR^d} \phi(y)  \overline A_p^*\rho(y) dy
+\bigo(h^{p+1}),
\end{equation*}
and we obtain from \eqref{eq:phiL},
\begin{eqnarray*}
\int_{\IR^d} \phi(y)  \overline A_p^*\rho(y) dy &=& - \int_0^\infty \int_{\IR^d} \mathcal{L}u(t,y)  \overline A_p^*\rho(y) dydt
= \int_0^\infty \int_{\IR^d} [\mathcal{L},\overline A_p] u(t,y) \rho(y) dydt
\\
&=& - \int_{\IR^d} A_p \phi(y)  \rho(y) dy =  - \lambda_p
\end{eqnarray*}
where $\lambda_p$ is defined in \eqref{eq:lambda}, and we used assumption \eqref{eq:asscommutator} and $\mathcal{L}^* \rho =0$ (see \eqref{eq:Lstar_rho}). 
We deduce
\begin{equation} 
%\nonumber
\IE \left(\int_{\IR^d} \phi(G_n(y)) \rho(y) dy\right) 
%&=& \int_{\IR^d} \phi(y) \rho(x) dx - h^p \int_{\IR^d} A_p \phi(x)  \rho(x) dx
%+\bigo(h^{p+1}),\\
= \int_{\IR^d} \phi(y) \rho(y) dy - h^p \lambda_p
+\bigo(h^{p+1}), \label{eq:proof1}
\end{equation}
We now apply Theorem \ref{thm:talay} where $\phi$ is replaced by $\phi\circ G_n = \phi + \bigo(h)$. 
We deduce using \eqref{eq:numexp} and properties of conditional expectations
that 
$\overline X_n = G_n(X_n)$ satisfies
$$
\left|\IE(\phi(\overline X_n)) - \IE \left(\int_{\IR^d} \phi(G_n(y)) \rho(y) dy\right) -  h^p\lambda_p \right|
\leq K(x)e^{-\lambda t_n} + Ch^{p+1}.
$$
This combined with \eqref{eq:proof1} concludes the proof of \eqref{eq:th2}
and analogously \eqref{eq:th1}. 
\end{proof}

%We mention two comments on the hypotheses of Theorem \ref{thm:main}.
\begin{remark}
Similarly to Remark \ref{rem:star}, observe in Theorem \ref{thm:main} that the (local) weak order $p$
assumption \eqref{weak:conv_loc} on the scheme $X_n$ can be replaced by the more general condition~\eqref{eq:general_condition}. 
\end{remark}

\begin{remark}
Notice that in the assumptions of Theorem \ref{thm:main}, the postprocessor map $G_n$ needs not neither to involve the same random variables as $X_n\mapsto X_{n+1}$ nor to be independent from these random variables. 
In particular, the postprocessors $\overline X_n$ in \eqref{eq:non-Markovian0},\eqref{eq:poststo},\eqref{eq:postdet} could be defined using other independent Gaussian random variables with law $\cN(0,I)$ instead of $\xi_n$. Our numerical investigations indicate that there is no benefit to generate additional random variables. 
\end{remark}

%\begin{remark}
%Notice that in the case where the postprocessor map $G_n=G$ is deterministic, the proof of Theorem \ref{thm:main}
%can be simplified as follows. 
%We consider the scheme $\overline X_{n+1} = G \circ \Phi_h \circ G^{-1}(\overline X_n)$ where
%$X_{n+1}=\Phi(X_n(h,X_n)$ is the original method. 
%We deduce from \eqref{eq:phiG} that $G$ is a 
%near identity map with inverse with inverse $G^{-1}$ satisfying 
%$$
%\phi(G^{-1}(x)) = \phi(x) - h^p\overline A_p \phi + \bigo(h^{p+1}).
%$$
%Using Assumption \ref{ass1} and properties of conditional expectations then yields
%\begin{eqnarray*}
%\IE(\phi(\overline X_1)|\overline X_0=x)&=& (y\mapsto \IE(\phi\circ G(X_1) | X_0=y) \circ G^{-1} (x)\\
%&=&(I-h^p\overline A_p)(y\mapsto \IE(\phi\circ G(X_1) | X_0=y) (x) + \bigo(h^{p+2})\\
%&=&(I-h^p\overline A_p)(I+h\mathcal{L}+h^2 A_1(f,g)+\ldots+h^{p+1} A_p) (\phi\circ G) (x) + \bigo(h^{p+2})\\
%&=&(I-h^p\overline A_p)(I+hA_0+h^2 A_1+\ldots+h^{p+1} A_p)(I+h^p\overline A_p) \phi (x) + \bigo(h^{p+2})\\
%&=& \IE(\phi( X_1)|X_0=x) + h^{p+1} [\mathcal{L},\overline A_p]+ \bigo(h^{p+2})
%\end{eqnarray*}
%We obtain that $\overline X_1$ has the same expansion as $X_1$ in  \eqref{eq:numin_taylor_expansion_formal}
%up to $\bigo(h^{p+2})$, with the exception that $A_p$ is replaced by $A_p+[\mathcal{L},\overline A_p]$.
%Noticing that $\overline X_n = G(X_n)$ if the initial values satisfy $X_0=G^{-1}(\overline X_0)$, 
%we obtain that $\overline X_n$ and $G(X_n)$ yield the same accuracy for the invariant measure. 
%\end{remark}

\subsection{Application to Brownian dynamics}
To illustrate the proposed approach, we focus on the Brownian dynamics model \eqref{eq:Brownian} where $f=-\nabla V$ is the gradient of a smooth potential. We assume\footnote{%
Notice that the item 2. of Assumption \ref{th:ergodic1} is automatically satisfied since $g(\xi)^Tg(\xi):\nabla^2 = \sigma^2 \Delta$ for the Brownian dynamics model \eqref{eq:Brownian}, where $\Delta$ is the Laplacian operator.
} Assumption \ref{th:ergodic1}, which guaranties the ergodicity of the system and the exponential estimate \eqref{eq:expergo}.

We now consider the following modification 
with negligible overhead of the $\theta$-method,
\begin{equation} \label{eq:theta}
X_{n+1} = X_n + h  (1-\theta)f(X_n + a \sigma \sqrt{h} \xi_n) 
+ h \theta f(X_{n+1} + a \sigma \sqrt{h} \xi_n) + \sigma \sqrt{h} \xi_n,
\end{equation}
where $a$ is a fixed scalar parameter, together with the postprocessor
\begin{equation} \label{eq:postproc}
\overline X_n = X_n  + b h ((1-\theta) f(X_n) + \theta f(\overline X_n)) + c \sigma \sqrt{h} \xi_n,
\end{equation}
where $\xi_n\sim\cN(0,I)$ are independent Gaussian random vectors with dimension $d$, and $b,c$ are fixed parameters.
Notice that for all $a$, the modified $\theta$-method \eqref{eq:theta} with $\theta\ne1/2$ has the same order one of accuracy for the invariant measure as the standard version \eqref{eq:thetastd}.
For $a=b=c=0$, we recover the standard $\theta$-method \eqref{eq:thetastd}
and $\overline X_n = X_n$.
The following Theorem provides order conditions on $a,b,c$ for $\overline X_n$ to achieve second order of accuracy for the invariant measure. 

\begin{theorem} \label{thm:thetameth}
Assume Assumption \ref{th:ergodic1} and
consider the scheme \eqref{eq:theta} and the postprocessor \eqref{eq:postproc} applied to Brownian dynamics \eqref{eq:Brownian}.
If 
\begin{equation}\label{eq:ordercond}
a^2+2\theta a + b-c^2 = 0,\qquad 
a  + b-c^2 -\frac14+\frac\theta 2=0,
\end{equation}
then, assuming the ergodicity of $\{X_n\}$, we have
\begin{eqnarray*}
\left| \lim_{N \rightarrow \infty} \frac{1}{N+1}\sum_{k=0}^{N} \phi(\overline X_{k})-\int_{\R^d}\phi(y)d\mu(y) \right| &\leq& C h^{2}
\\
\left|\IE(\phi(\overline X_n)) - \int_{\IR^d}\phi(y)d\mu(y) \right|
&\leq& K(x)e^{-\lambda t_n} + Ch^{2},
\end{eqnarray*}
for all $\phi\in\CC^\infty(\IR^d,\IR)$ with $t_n=nh$, where $C,K(x)$ are independent of $n$ and $h$ assumed small enough. 
In particular, assuming ergodicity, the postprocessed schemes \eqref{eq:non-Markovian0},
\eqref{eq:theta1}-\eqref{eq:poststo}, and \eqref{eq:theta1}-\eqref{eq:postdet}
satisfy the above estimates.
\end{theorem}
The proof of Theorem \ref{thm:thetameth} is postponed to the Appendix.
We explain here why it yields the order two of accuracy for the invariant measure of the postprocessed schemes \eqref{eq:non-Markovian0}, \eqref{eq:theta1}-\eqref{eq:poststo}, and \eqref{eq:theta1}-\eqref{eq:postdet}.

If $\theta=0$ or $\theta\geq 1/2$, and for a given $c$, the order conditions \eqref{eq:ordercond} can be shown to have real solutions given by
$$
a=\frac12-\theta \pm\frac{\sqrt{2\theta(2\theta-1)}}2, \quad
b= c^2-\frac14+\frac\theta2  \mp\frac{\sqrt{2\theta(2\theta-1)}}2.
$$
We now focus on special values of $\theta$ to construct 
various methods of order two for the invariant measure of \eqref{eq:Brownian}.
\begin{itemize}
\item 
For
$\theta=0$ (Euler-Maruyama method), we obtain 
$a=\frac14, b=c^2-\frac 14$. Imposing respectively $b=0$ and $c=0$ yields
the following postprocessed methods of order two:
the method \eqref{eq:non-Markovian0} from \cite{LM13},
$$
X_{n+1} = X_n + h  f(X_n + \frac12 \sigma \sqrt{h} \xi_n) 
 + \sigma \sqrt{h} \xi_n,
\qquad
\overline X_n=  X_n + \frac12 \sigma \sqrt{h} \xi_n.
$$
and the variant
\begin{equation*} %\label{eq:non-Markovian1}
X_{n+1} = X_n + h  f(X_n + \frac12 \sigma \sqrt{h} \xi_n) 
 + \sigma \sqrt{h} \xi_n,
\qquad
\overline X_n=  X_n + \frac {h}4  f(X_n) .
\end{equation*}
\item 
For $\theta=\frac12$, we obtain $a=0$ and $b=c^2$. This is not surprising
because the standard $\theta$-method \eqref{eq:thetastd} with $\theta=\frac12$ is already exact for sampling the invariant measure as noted in \cite{Sch99b},
and it has weak order two when applied to SDEs with additive noise such as \eqref{eq:Brownian}. Postprocessors for order two are therefore useless in this case.
\item 
For $\theta=1$, the assumption $b=0$ yields the method \eqref{eq:theta1},\eqref{eq:poststo} (without the matrix $J_n^{-1}$) of the form
$$
X_{n+1} = X_n + h  f(X_{n+1} + a \sigma \sqrt{h} \xi_n) 
 + \sigma \sqrt{h} \xi_n,\quad
\overline X_n=  X_n + c \sigma \sqrt{h} \xi_n.
$$
Since the matrix $J_n^{-1}$ in \eqref{eq:poststo} is a $\bigo(h)$ perturbation of the identity, its does not change the order $p=1$ weak expansion of the postprocessor
in assumption \eqref{eq:phiG} of Theorem \ref{thm:main}.
The method \eqref{eq:theta1},\eqref{eq:poststo} therefore has order two of accuracy for the invariant measure. 

Imposing alternatively the condition of order $3$ for linear problems (i.e. \eqref{eq:condlin} with $p=3$) yields the method \eqref{eq:theta1},\eqref{eq:postdet} of the form
$$
X_{n+1} = X_n + h  f(X_{n+1} + a \sigma \sqrt{h} \xi_n) 
 + \sigma \sqrt{h} \xi_n,\quad
\overline X_n=  X_n + h b f(\overline X_{n})  + c \sigma \sqrt{h} \xi_n,
$$
which has order two of accuracy for the invariant measure, and order three for linear problems (i.e. for a quadratic potential $V$). 
\end{itemize}

\begin{figure}[htb]
\centering
\bigskip
\scalebox{0.82}{\global\def\path{#1}\input{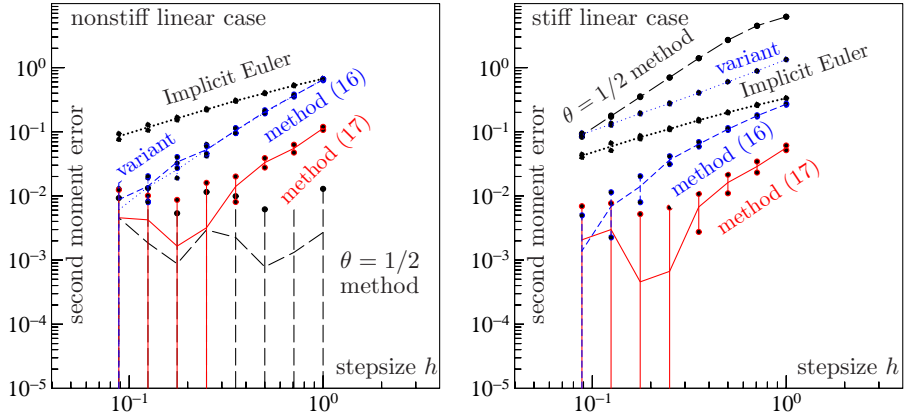}}
\bigskip
\caption{
Linear problem with potential \eqref{eq:pb1}.
Implicit Euler method (\eqref{eq:thetastd} with $\theta=1$, order 1),
$\theta=1/2$-method (order 2), and postprocessed versions (\eqref{eq:theta1},\eqref{eq:poststo} and \eqref{eq:theta1},\eqref{eq:postdet}) of order $2$.
The variant of the method \eqref{eq:theta1},\eqref{eq:poststo} without the stabilization term $J_n^{-1}$ is also included.
Error for the second moment $\int_\R (x_1^2+x_2^2) \rho(x) dx$ versus time stepsize $h$ obtained using $10$ 
trajectories on a long time interval of length $T=10^5$. 
The vertical bars indicate the standard deviation intervals due to the Monte-Carlo error. 
\label{fig:figconv1}}
\end{figure}

\section{Numerical experiments}
\label{sec:num}

In this section, we illustrate numerically the convergence properties
of the new implicit schemes introduced in Section \ref{sec:meth} for the high order sampling of the invariant measure of Brownian dynamics \eqref{eq:Brownian}. 
Although our analysis applies only to globally Lipschitz vector fields, we shall consider numerical examples with non-globally Lipschitz fields, since such a regularity often arises in applications. 
We highlight once again that the use of the proposed schemes can still be made rigorous in this context, using the concept of rejecting exploding trajectories proposed in \cite{MT05,MT07}.

We focus on stiff problems for which implicit methods are needed to avoid the severe step size restrictions of standard explicit integrators. 
We consider the standard implicit Euler method (\eqref{eq:thetastd} with $\theta=1$), the standard $\theta=1/2$ method \eqref{eq:thetastd}, and the
new postprocessed methods \eqref{eq:theta1},\eqref{eq:poststo}
and \eqref{eq:theta1},\eqref{eq:postdet}.
We have implemented all methods in Fortran, using a Newton method with tolerance $10^{-13}$ to solve the nonlinear implicitness of the methods. 
The results for the explicit postprocessed scheme \eqref{eq:non-Markovian0} from \cite{LM13}
are not included because a severe stepsize
restriction $h\ll 1$ makes explicit methods unusable for such stiff problems (analogously to the deterministic case).

\begin{figure}[bt]
\centering
%\bigskip
\begin{multicols}{2}
Nonstiff case \eqref{eq:pb2nonstiff}\\
\includegraphics[scale=0.5]{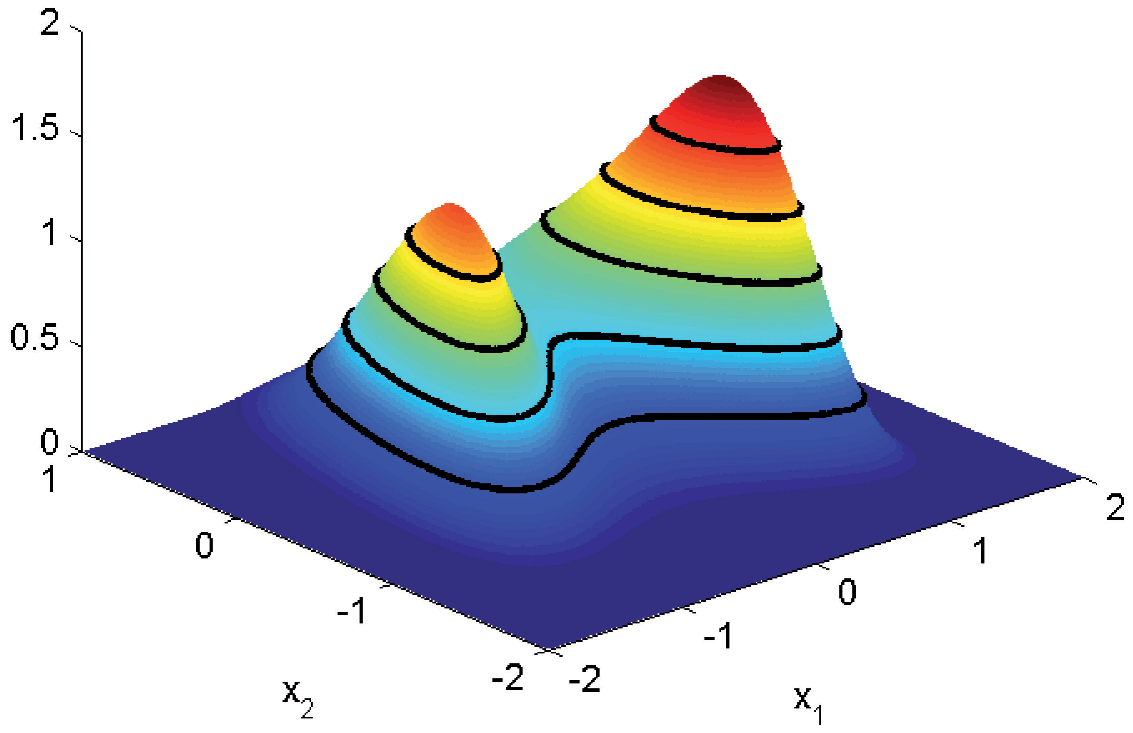}\\[-3mm]
\includegraphics[scale=0.5]{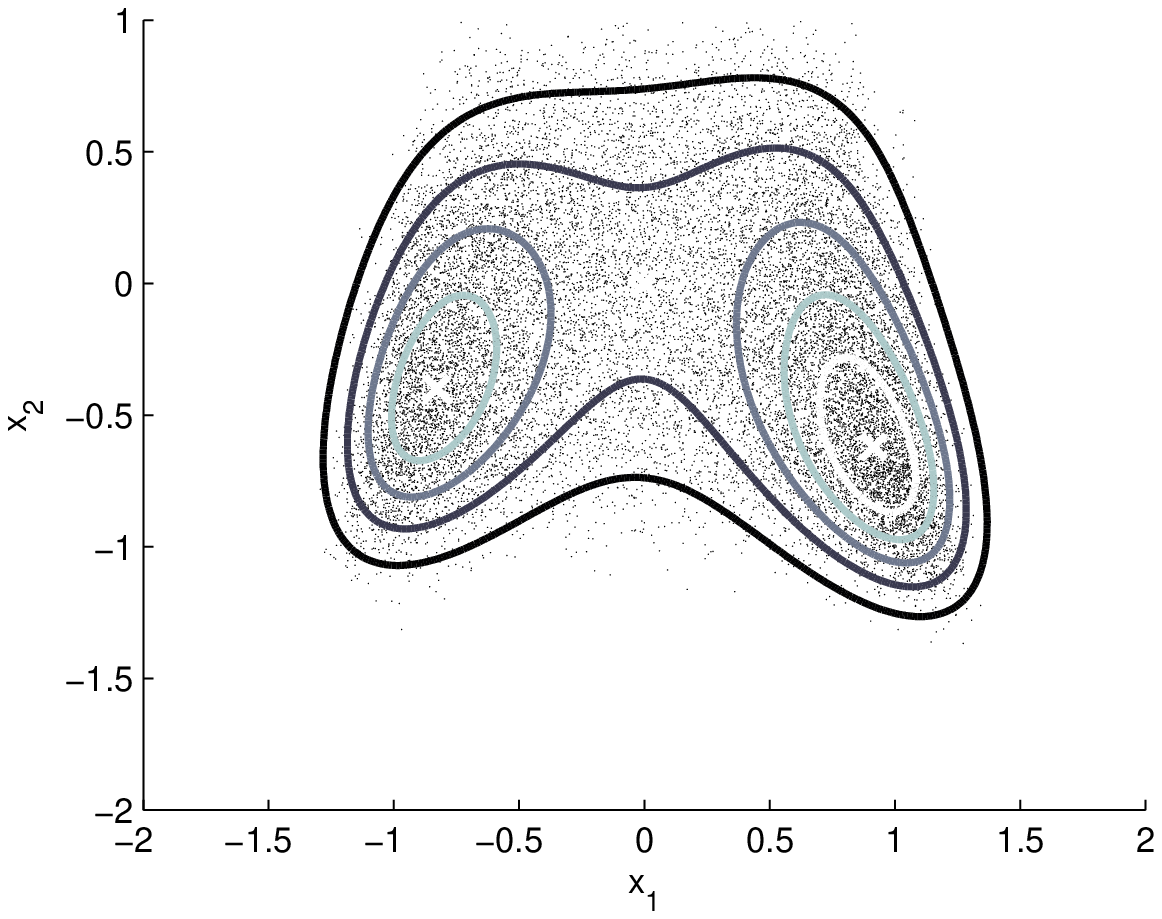}\\
\columnbreak
Stiff case \eqref{eq:pb2stiff}\\
\includegraphics[scale=0.5]{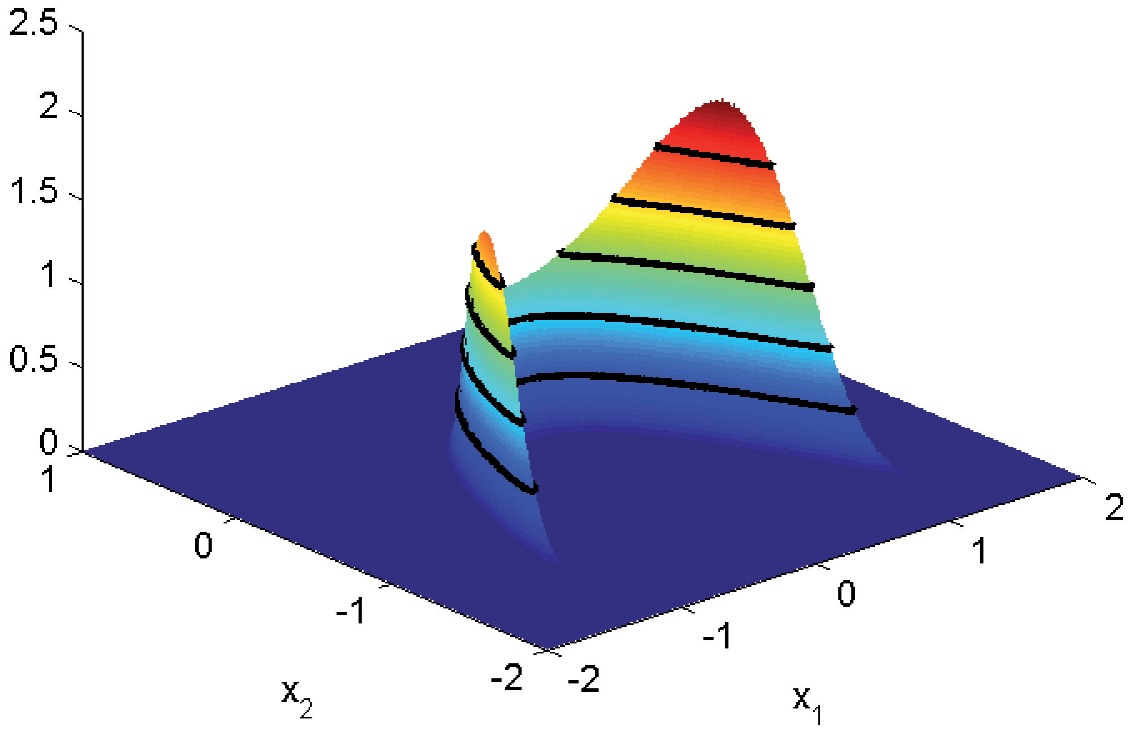}\\[-3mm]
\includegraphics[scale=0.5]{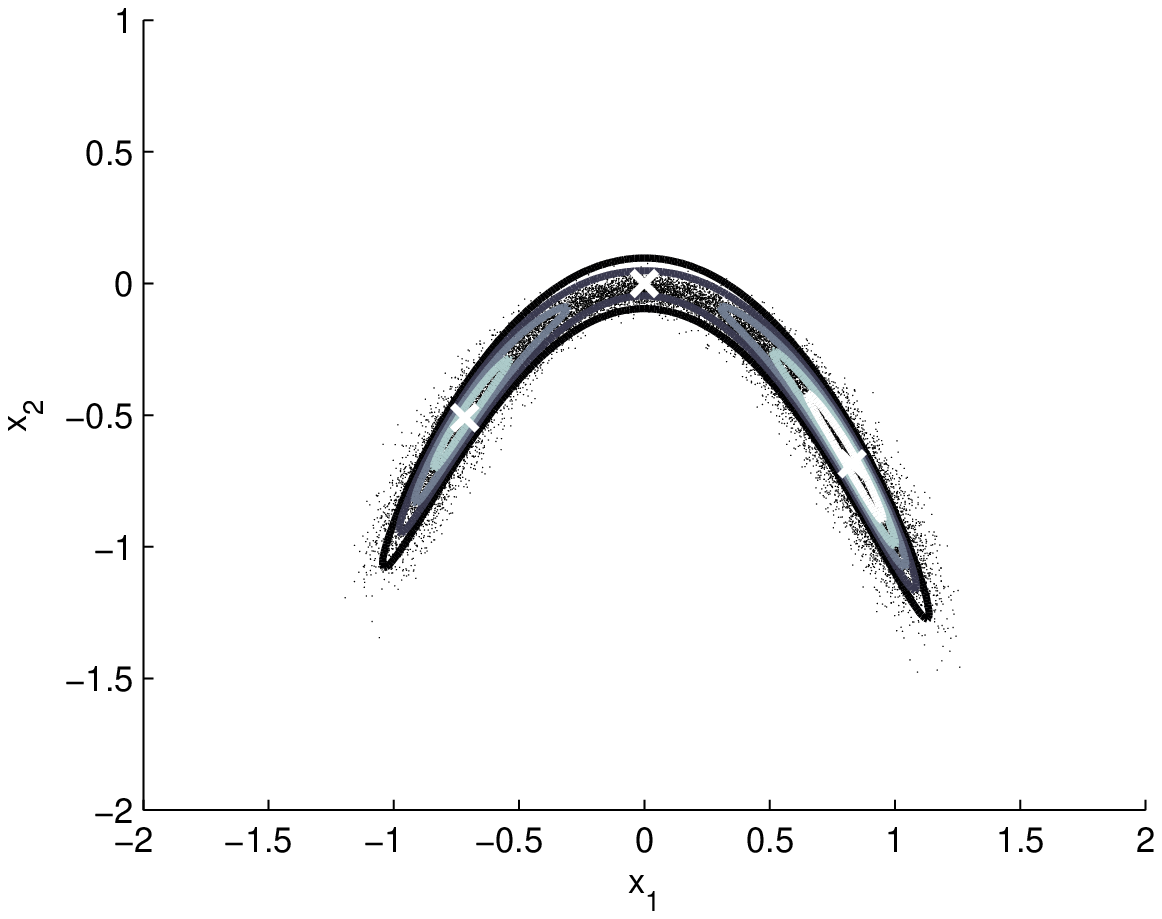} 
\end{multicols}
%\vspace{-3mm} 
%\includegraphics[scale=0.58]{prog/figs/levelnonstiff3d.eps}
%\hskip-4mm
%\includegraphics[scale=0.58]{prog/figs/levelstiff3d.eps} \\
%\includegraphics[scale=0.58]{prog/figs/levelnonstiff.eps}
%\hskip-4mm
%\includegraphics[scale=0.58]{prog/figs/levelstiff.eps} 
%\vspace{-3mm}
\caption{
Plots of the Gibbs density \eqref{eq:rhoexp0} for the nonlinear problems with potential \eqref{eq:pb2nonstiff} (nonstiff case, left pictures) and  \eqref{eq:pb2stiff} (stiff case, right pictures).
Four level curves are represented in solid lines, the three extrema are represented with crosses.  
In the bottom pictures, a numerical trajectory of the method \eqref{eq:theta1} is included, with length $T=10^3$ and stepsize $h=0.05$.
\label{fig:figrk}}
\end{figure}

\subsection{Linear test problem}
We consider first the linear problem in 2D, with a quadratic potential defined by
\begin{equation} \label{eq:pb1}
V(x)= \frac{x_1^2}{2} + \frac{x_2^2}{2\varepsilon} + \frac{x_1x_2}4.
\end{equation}
For the arbitrary initial condition $x_1=x_2=-3$, we plot in Figure
\ref{fig:figconv1} the error for the second moment $\int_\R (x_1^2+x_2^2) \rho(x) dx$ for many time stepsizes $h$, using the average over $10$ long trajectories of length $10^5$. 
We compare in Figure \ref{fig:figconv1} the nonstiff case ($\eps=1$) and the stiff case ($\eps=10^{-6}$).
\begin{itemize}
\item
In the nonstiff case (left picture of Figure \ref{fig:figconv1}), we observe the expected order 2 for the postprocessed
method \eqref{eq:theta1},\eqref{eq:poststo} (we also include the variant without 
matrix $J_n^{-1}$ in \eqref{eq:poststo}, which yields nearly identical results), while 
the postprocessed method \eqref{eq:theta1},\eqref{eq:postdet} has order 3 for linear problems and is significantly more accurate by a factor of about 10.
A line of slope 1 can be observed for the standard implicit Euler method.
The $\theta=1/2$ method is exact for linear Brownian dynamics and the observed bias with magnitude $10^3$ is due to the Monte-Carlo error (known \cite{MSH02} to decay as $\bigo(T^{-1/2})$).

\item
In the stiff case (right picture of Figure \ref{fig:figconv1}), the implicit Euler method and the postprocessed methods \eqref{eq:theta1},\eqref{eq:poststo} and \eqref{eq:theta1},\eqref{eq:postdet} do not loose their efficiency. 
In contrast, the $\theta=1/2$ method becomes poorly accurate in the very stiff regime. This phenomena is well known already for severely stiff deterministic problems (the $\theta=1/2$ method is not $L$-stable), and this is highlighted in the stochastic context in \cite{LAE08}. A loss of accuracy is also observed for the variant of the method \eqref{eq:theta1},\eqref{eq:poststo} where the term $J_n^{-1}$ providing $L$-stability is omitted in \eqref{eq:poststo}. 
\end{itemize}

\subsection{Nonlinear test problems}

We next consider two nonlinear problems.
On the one hand, we consider a nonstiff problem in dimension $d=2$ with potential
\begin{equation} \label{eq:pb2nonstiff}
V(x) = (1-x_1^2)^2+x_2^4- x +x_1\cos(x_2)  +(x_2+x_1^2)^2,
\end{equation}
and on the other hand, we consider the stiff problem  in dimension $d=3$ with potential
\begin{equation} \label{eq:pb2stiff}
V(x) = (1-x_1^2)^2+x_2^4- x +x_3\cos(x_2)  +100(x_2+x_1^2)^2 + \frac{10^6}2(x_1-x_3)^2.
\end{equation}
\begin{figure}[bt]
\centering
\bigskip
\scalebox{0.82}{\global\def\path{#1}\input{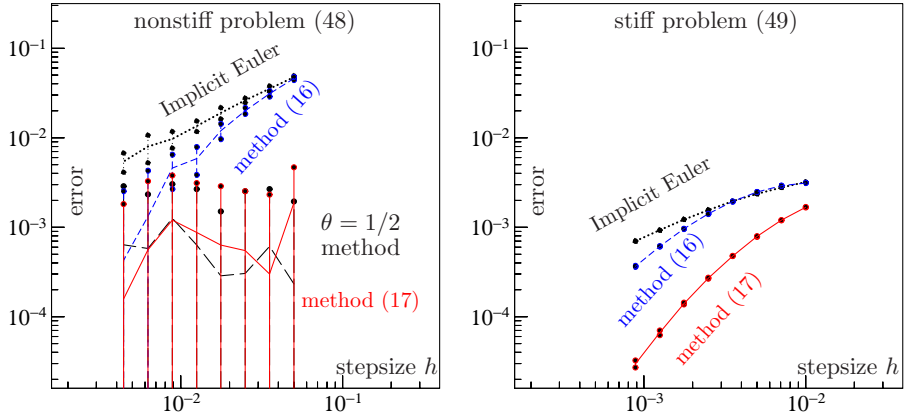}}
\bigskip
\caption{
Nonlinear nonstiff problem \eqref{eq:pb2nonstiff} and stiff problem\eqref{eq:pb2stiff}. 
Implicit Euler method (\eqref{eq:thetastd} with $\theta=1$, order 1) and postprocessed versions \eqref{eq:theta1},\eqref{eq:poststo} and \eqref{eq:theta1},\eqref{eq:postdet} of order $2$.
The $\theta=1/2$-method (order 2) is also included in the nonstiff case (left picture).
Error in $\int_\R^d (x_2+x_1^2)^2 \rho(x) dx$ versus time stepsize $h$ obtained using $10$ 
trajectories on a long time interval of length $T=10^5$. 
The vertical bars indicate the standard deviation intervals due to the Monte-Carlo error. 
\label{fig:figconv2}}
\end{figure}
In Figure \ref{fig:figrk}, we plot the associated Gibbs density measure \eqref{eq:rhoexp0}
together with five level curves for the components\footnote{%
Since \eqref{eq:pb2stiff} is in 3D, the right pictures of Figure \ref{fig:figrk} correspond to the Gibbs density measure for the marginal law 
of the components $(x_1,x_2)$.
} $x_1,x_2$ for the nonstiff problem \eqref{eq:pb2nonstiff}  (left pictures)
and the stiff problem \eqref{eq:pb2stiff} (right pictures). 
It can be seen that the distribution is much narrower in the stiff case
and concentrated close to the curve $\{x_2+x_1^2=0\}$, due to the stiff term $100(x_2+x_1^2)^2$ in \eqref{eq:pb2stiff}. 
In the bottom pictures of Figure \ref{fig:figrk}, we also include one numerical trajectory of length $T=10^3$ with stepsize $h=0.05$ for the $L$-stable method \eqref{eq:theta1}. It fits well the invariant distribution for both the nonstiff and stiff cases.

For both problems, we take the initial conditions $x_i=i$ ($i=1,2,3$) and consider ten long trajectories with length $T=10^4$.
Analogously to the linear case, we compare in Figure \ref{fig:figconv2} the results in the nonstiff and stiff cases. We plot for many constant stepsizes $h$ the errors for the evaluation of the quantity 
$\int_{\R^d} (x_2+x_1^2)^2 \rho(x) dx$ (a reference value for this integral is computed using an accurate quadrature method).

\begin{itemize}
\item In the nonstiff case (left picture of Figure \ref{fig:figconv2}), we observe the expected convergence lines of slope 1 and 2 for the implicit Euler method and the variant 
postprocessed method \eqref{eq:theta1},\eqref{eq:poststo}, while
the $\theta=1/2$ method and the postprocessed method \eqref{eq:theta1},\eqref{eq:postdet} (order 2 for nonlinear problems, order 3 for linear problems) reveal much more accurate.
\item In the stiff case (right picture of Figure \ref{fig:figconv2}),
the Newton iterations of the $\theta=1/2$-method fail to converge for the considered stepsizes, a convergence curve is thus not included.
The implicit Euler method still exhibits a convergence curve $1$.
The best accuracy is observed for the postprocessed method \eqref{eq:theta1},\eqref{eq:postdet} of order two. 
\end{itemize}

\begin{figure}[bt]
\begin{multicols}{2}
%Nonstiff case\\
\includegraphics[scale=0.5]{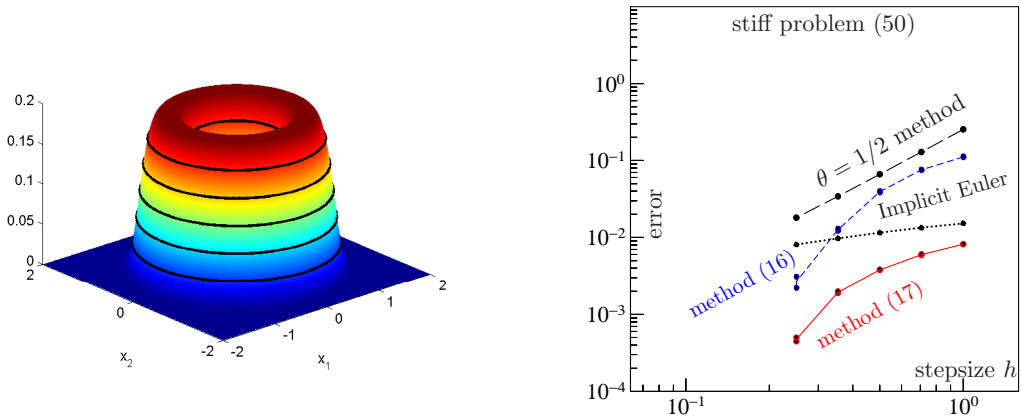}\\
\columnbreak
\hskip8mm \scalebox{0.82}{\global\def\path{#1}\input{prog/figs/figconv3.inp}}
\end{multicols}

%\centering
%\bigskip
%\GGGinput{prog/figs/}{figconv3}
%\bigskip
\caption{
Nonlinear problem in dimension $d=10$ with potential \eqref{eq:pb3}.
Left picture: plot of the Gibbs density \eqref{eq:rhoexp0} in the case $d=2$.
Right picture: comparison of the implicit Euler method (\eqref{eq:thetastd} with $\theta=1$, order 1) and postprocessed versions (new method \eqref{eq:theta1},\eqref{eq:postdet} and variant \eqref{eq:theta1},\eqref{eq:poststo}) of order $2$.
The $\theta=1/2$-method (order 2) is also included.
Error for $\int_{\R^d} (|x|-1)^2 \rho(x) dx$ versus time stepsize $h$ obtained using $10$ 
trajectories on a long time interval of length $T=10^4$. 
%The vertical bars indicate the standard deviation intervals. 
\label{fig:figconv3}}
\end{figure}

Finally, as a more challenging problem, we consider the following potential in dimension $d=10$,
\begin{equation} \label{eq:pb3}
V(x) = 25(1-|x|)^4, \qquad |x|=\sqrt{\sum_{i=1}^d x_i^2}.
\end{equation}
This is a modification of a standard nonlinear model for the motion of a spring in $d$ dimensions considered as an illustrative expample in \cite[Sect.\ts7]{San14}.
We plot in Figure \ref{fig:figconv3} (left picture) the Gibbs density measure \eqref{eq:rhoexp0} in dimension $d=2$, and it can be seen that the distribution is concentrated close to the sphere $ \{|x|=1\}$
centred at the origin with radius $1$.
In Figure \ref{fig:figconv3} (right picture), we plot the convergence curves
for $\int_{\R^d} (|x|-1)^2 \rho(x) dx$. Again, excellent performances of the postprocessed method \eqref{eq:theta1},\eqref{eq:postdet} can be observed. For the considered stepsizes, it is about 20 times more accurate than the standard $\theta=1/2$ method, also of order two for the invariant measure. 
%Again, the variant \eqref{eq:theta1},\eqref{eq:poststo} exhibits performances for this problem.
As a reference solution, we took the result of the $\theta=1/2$ method with stepsize $h=2^{-4.5}\simeq 4.4 \cdot 10^{-2}$.

\section{Summary}

In this paper, we extended to the stochastic context the methodology of effective order of Butcher. The approach of stochastic postprocessed integrators was illustrated in the case of Brownian dynamics, for which new schemes of order two were constructed based on the $\theta$-method, in particular implicit schemes with favourable stability properties for the sampling of the invariant measure of ergodic systems. 
%The postprocessing allows for the considered examples to increase by orders of magnitude the accuracy of the standard $\theta$-method.
We rediscovered the second order non-Markovian scheme from \cite{LM13} and analyzed in \cite{LMT14} which can be interpreted as a postprocessed integrator. 
This new insight permitted us to construct new high order integrators based on the $\theta$-method, and to provide a rigorous error analysis based on the standard theory of Markovian schemes.
Our approach is in principle not restricted to Brownian dynamics and could apply to larger classes of SDEs (Theorem \ref{thm:main}). This will be investigated in future works.\looseness-3

%In this paper, we extended to the stochastic context the methodology of effective order of Butcher. The approach of stochastic postprocessed integrators was illustrated in the case of Brownian dynamics, for which new schemes of order two were constructed based on the $\theta$-method, in particular implicit schemes with favourable stability properties for stiff ergodic systems. 
%%The postprocessing allows for the considered examples to increase by orders of magnitude the accuracy of the standard $\theta$-method.
%%We showed that the second order non-Markovian scheme from \cite{LM13} (see Remark \ref{rem:leimkhuler}) can be interpreted as a postprocessed integrator. 
%%This new insight permitted us to construct new high order integrators based on the $\theta$-method for the sampling of the invariant measure of ergodic systems.
%Our approach is in principle not restricted to Brownian dynamics (Theorem \ref{thm:main}) and could apply to larger classes of SDEs. This will be investigated in future works. 

\bigskip

\noindent \textbf{Acknowledgements.}\
This work was
partially supported by the Fonds National Suisse, project No. 
200020\_144313/1.

\section{Appendix}

%\begin{proof}[Proof of Theorem \ref{thm:thetameth}]
\textit{Proof of Theorem \ref{thm:thetameth}.}
For the scheme \eqref{eq:theta}, a calculation analogous to \eqref{eq:A1std} shows that \eqref{eq:numin_taylor_expansion_formal} holds with
\begin{eqnarray*}%\label{eq:A1}
A_1\phi &=& \frac12 \phi''(f,f) + \frac{\sigma^2}2 \sum_{i=1}^d \phi'''(e_i,e_i, f)
+ \frac{\sigma^4}8  \sum_{i,j=1}^d \phi^{(4)} (e_i,e_i,e_j,e_j) +\theta \phi' f'f\nonumber\\
&+&  \Big((1-\theta)a^2+\theta(a+1)^2\Big) \frac{\sigma^2}2\phi' \sum_{i=1}^d f''(e_i,e_i)
+ {\Big(a +\theta\Big) } \sigma^2\sum_{i=1}^d \phi''(f'e_i,e_i).
\end{eqnarray*}
Using integration by parts, %and the identity $\nabla \rho = \frac2{\sigma^2} f \rho$, 
the following identities where proved in \cite{AVZ14a},
\begin{eqnarray} 
\left\langle \phi''(f,f) \right\rangle &=&   \textstyle
\left\langle   -\phi'(f'f + \ddiv (f) f + \frac2{\sigma^2} \|f\|^2 f) \right\rangle, 
\nonumber\\
\textstyle\left\langle \sigma^2 \sum_i \phi'''(f,e_i,e_i) \right\rangle    &=&  \textstyle
\left\langle  \phi'(\sigma^2  \sum_i f''(e_i,e_i) + 4f'f + 2 \ddiv (f) f + \frac4{\sigma^2} \|f\|^2 f) \right\rangle,\nonumber\\
\textstyle \left\langle \sigma^2\sum_{ij}\phi^{(4)}(e_i,e_i,e_j,e_j) \right\rangle &=&  \textstyle
\left\langle  -\sum_i 2\phi'''(f,e_i,e_i) \right\rangle,\nonumber\\
\textstyle \left\langle   \sigma^2 \sum_i \phi''(f'e_i,e_i) \right\rangle &=&  \textstyle
\left\langle  -\phi'(\sigma^2 \sum_i f''(e_i,e_i) + 2f' f) \right\rangle,
\label{eq:estimparts}
\end{eqnarray}
where we use the notation 
$\left\langle u \right\rangle = \int_{\IR^d} u(x) \rho(x) dx$ 
and the sums are for $i,j=1,\ldots,d$.
For instance, to prove the first identity in \eqref{eq:estimparts}, we integrate by parts and use
$\partial_j\rho=\frac{2}{\sigma^2}f^j \rho$, where $\partial_j$ denotes the derivative $\partial/(\partial x_j)$,
\begin{eqnarray*}
\left\langle \phi''(f,f) \right\rangle &=& \int_{\IR^d} \sum_{i,j=1}^d
\partial_i\partial_j \phi f^i f^j
 \rho dy
=
-\int_{\IR^d} \sum_{i,j=1}^d
\partial_i \phi \partial_j(f^if^j
 \rho) dy\\
&=&-\left\langle  \sum_{i,j=1}^d
\partial_i \phi (f^i\partial_jf^j
+\partial_jf^if^j
 +f^if^j
 \frac{2}{\sigma^2}f^j) 
\right\rangle\\
&=&\left\langle  
-\phi' 
(f'f + \ddiv (f) f + \frac2{\sigma^2} \|f\|^2 f) 
\right\rangle.
\end{eqnarray*}
The other identities in \eqref{eq:estimparts}  can be proved analogously.
We deduce
\begin{eqnarray*}
\left\langle A_1\phi\right\rangle
&=& \left\langle\Big(a^2+2\theta a\Big) \frac{\sigma^2}2\phi' \sum_{i=1}^d f''(e_i,e_i)
+ {\Big(-\frac14 + a +\frac12\theta\Big) } \sigma^2\sum_{i=1}^d \phi''(f'e_i,e_i) \right\rangle
\end{eqnarray*}
To improve the accuracy of the scheme \eqref{eq:theta}, 
we now consider the postprocessor \eqref{eq:postproc} 
where $b$ and $c\ge 0$ are fixed parameters. 
Using $\overline A_1 \phi= b\phi' f + \frac{c^2}2 \sigma^2 \Delta \phi$, we have
$$
[\mathcal{L},\overline A_1]\phi = \frac{(b-c^2)}2 \sigma^2\phi' \sum_{i=1}^d f''(e_i,e_i)
+ {(b-c^2) \sigma^2} \sum_{i=1}^d \phi''(f'e_i,e_i)
$$
Summing up then yields the value of $\left\langle A_1\phi + [\mathcal{L},\overline A_1]\phi\right\rangle$,
\begin{equation*}
\left\langle\Big(a^2+2\theta a + b-c^2\Big) \frac{\sigma^2}2\phi' \sum_{i=1}^d f''(e_i,e_i)
+ {\Big(-\frac14 + a +\frac12\theta + b-c^2\Big) } \sigma^2\sum_{i=1}^d \phi''(f'e_i,e_i) \right\rangle
\end{equation*}
and this quantity is zero if \eqref{eq:ordercond} holds.
We conclude the proof of Theorem \ref{thm:thetameth} using Theorem \ref{thm:main}.
\qed%\end{proof}

\bibliographystyle{abbrv}
\bibliography{abd_biblio,complete1,complete,HLW}

\end{document}